\documentclass[reqno,12pt]{amsart}
\title[Sofic groups]{Free products of sofic groups with amalgamation over monotileably amenable groups}
\author[Collins, Dykema]{Beno\^\i{}t Collins$^{\dagger}$}
\address{B.\ Collins \\
Department of Mathematics and Statistics \\
University of Ottawa \\
585 King Edward \\
Ottawa, ON K1N 6N5 Canada \\
and \\
CNRS \\
Department of Mathematics \\
Lyon 1 Claude Bernard University, France} 
\email{bcollins@uottawa.ca}
\thanks{\footnotesize $^{\dagger}$Research supported in part by NSERC
Discovery grant RGPIN/341303-2007 and the ANR GranMa and Galoisint grants.}

\author{Kenneth J.\ Dykema$^{*}$}
\address{K.\ Dykema \\
Department of Mathematics \\
Texas A\&M University \\
College Station, TX 77843-3368, USA}
\email{kdykema@math.tamu.edu}
\thanks{\footnotesize $^{*}$Research supported in part by NSF grant DMS-0901220}



\usepackage{amscd,amssymb,comment,epic,eepic,euscript,graphics}
\usepackage{enumerate}
\usepackage[initials]{amsrefs}

\theoremstyle{plain}
\newtheorem{thm}{Theorem}[section]
\newtheorem{cor}[thm]{Corollary} 
\newtheorem{lemma}[thm]{Lemma} 
\newtheorem{prop}[thm]{Proposition}

\theoremstyle{remark}

\newtheorem{remark}[thm]{Remark}

\theoremstyle{definition}

\newcommand\Cpx{{\mathbf C}}
\newcommand\dist{\operatorname{dist}}
\newcommand\Eb{{\mathbf E}}

\newcommand\eps{{\epsilon}}

\newcommand\id{{\mathrm{id}}}
\newcommand\Ints{{\mathbf Z}}

\newcommand\Nats{{\mathbf N}}
\newcommand\Pb{{\mathbf P}}
\newcommand\Pc{{\mathcal{P}}}
\newcommand\simp{\overset{p}{\sim}}
\newcommand\phit{{\tilde\phi}}

\newcommand\pt{{\tilde p}}
\newcommand\restrict{{\upharpoonright}}

\newcommand\Sym{{\operatorname{Sym}}}
\newcommand\tr{{\mathrm{tr}}}
\newcommand\Tr{{\mathrm{Tr}}}

\newcommand\Vt{{\widetilde V}}

\begin{document}

\begin{abstract}
We show that free products of sofic groups with amalgamation over monotileably amenable subgroups are sofic.
Consequently, so are HNN extensions of sofic groups relative to homomorphisms of monotileably amenable subgroups.
We also show that families of independent uniformly distributed permutation matrices and certain families
of non--random permutation matrices (essentially, those
coming from quasi--actions of a sofic group)
are asymptotically $*$--free as the matrix size grows without bound.
\end{abstract}

\maketitle

\section{Introduction}

Sofic groups were introduced by M.\ Gromov~\cite{G99}
and named by B.\ Weiss~\cite{W00}.
In short, a group is sofic if it can be approximated (in a certain weak sense) by permutations.
All amenable and residually amenable groups are sofic.
Due in large part to work of Elek and Szab\'o~\cite{ES06}, the class of sofic groups is known to
be closed under taking
direct products, subgroups, inverse limits, direct limits, free products, and extensions by amenable groups.
See also~\cite{T} and~\cite{C} for recent interesting examples.
It is unknown whether all groups are sofic, though Gromov's famous paradoxical dictum
(``any statement about all countable groups is either trivial or false'') would argue against it.

Several results illustrate the utility of knowing that a given group is sofic.
Gromov~\cite{G99} proved that Gottschalk's Surjunctivity Conjecture holds for the groups now called sofic.
Elek and Szab\'o~\cite{ES04} proved that Kaplansky's Direct Finiteness Conjecture holds for sofic groups.
In~\cite{ES05} they gave a description of sofic
groups in terms of ultrapowers and proved that sofic groups are hyperlinear, which entails that their group von Neumann algebras
embed in $R^\omega$;
thus, the topic of sofic groups makes contact with Connes' Embedding Problem, which is a fundamental open
problem in the theory of von Neumann algebras.
See the survey articles~\cite{P08} and~\cite{PK} for more on hyperlinear and sofic groups.
A.\ Thom~\cite{T08} proved some interesting results about the group rings of sofic groups.
L.\ Bowen~\cite{Bo10} classified the Bernoulli shifts of a sofic group,
provided that the group is also Ornstein (e.g., if it contains an infinite amenable group as a subgroup).

\medskip

Now we recall a few basic notions and give a definition of sofic groups.
(See~\cite{ES04} for a proof that the definition in~\cite{W00}, which was for finitely generated groups, agrees with the one found below
if the group is finitely generated.)
The {\em normalized Hamming distance} $\dist(\sigma,\tau)$ between two permutations $\sigma$ and $\tau$,
both elements of the symmetric group $S_n$, is defined to be the number
of points not fixed by $\sigma^{-1}\tau$, divided by $n$.
Note that if we consider $S_n$ as acting on an $n$--dimensional complex vector space 
as permutation matrices, then this normalized
Hamming distance is equal to $1-\tr_n(\sigma^{-1}\tau)$, where $\tr_n$ is the trace on $M_n(\Cpx)$ normalized
so that the identity has trace $1$.

A group $\Gamma$ is {\em sofic} if for every finite subset $F$ of $\Gamma$
and every $\eps>0$,
there exist an integer $n\ge1$ and a map $\phi:\Gamma\to S_n$ such that
\begin{enumerate}[(i)]
\item
for every $g\in F\backslash\{e\}$, $\dist(\phi(g),\id)>1-\eps$, where $e$ is the identity element of $\Gamma$,
\item
for all $g_1,g_2\in F$, $\dist(\phi(g_1^{-1}g_2),\phi(g_1)^{-1}\phi(g_2))<\eps$.
\end{enumerate}
We will call a map $\phi$ satisfying these properties an $(F,\eps)$--{\em quasi--action} of $\Gamma$.

Since a group is sofic if and only if all of its finitely generated subgroups are sofic, it will
suffice to consider countable groups, and it will be convenient to have the 
elementary reformulation of soficity contained in the following proposition, whose proof is an easy exercise.
Given positive integers $n(k)$, we let $\bigoplus_{k=1}^\infty(S_{n(k)},\dist)$
denote the normal subgroup of $\prod_{k=1}^\infty S_{n(k)}$ consisting of all sequences $(\sigma_k)_{k=1}^\infty$
such that $\lim_{k\to\infty}\dist(\sigma_k,\id_{n(k)})=0$,
where $\id_{n(k)}$ is the identity element of the permutation group $S_{n(k)}$.
\begin{prop}\label{prop:soficlift}
Let $\Gamma$ be a countable group.
Then 
$\Gamma$ is sofic if and only if 
for some sequence of positive integers $n(k)$, there is a group homomorphism
\[
\psi:\Gamma\to\left(\prod_{k=1}^\infty S_{n(k)}\right)\bigg/\left(\bigoplus_{k=1}^\infty(S_{n(k)},\dist)\right),
\]
given by $\psi(g)=[(\psi_k(g))_{k=1}^\infty]$ for some maps $\psi_k:\Gamma\to S_{n(k)}$ so that
\[
\lim_{k\to\infty}\dist(\psi_k(g),\id_{n(k)})=1
\]
for all nontrivial elements $g$ of $\Gamma$.
\end{prop}

\medskip

In this paper, we prove that the class of sofic groups is closed under taking free products with amalgamation over
monotileably amenable subgroups.
Recall that a group $G$ is amenable if and only if for every finite set $K$ and every $\eps>0$, there is
a $(K,\eps)$--invariant set, namely, a finite set $F\subseteq G$ such that $|KF\backslash F|<\eps|F|$.
A {\em tile} (or monotile) for a group $G$ is a finite set $T\subseteq G$ such that
$G$ is a disjoint union of right translates of $T$.
We may chose a set $C\subseteq G$ of {\em centers}, so that the map $T\times C\to G$
given by multiplication $(t,c)\mapsto tc$ is a bijection.
Clearly, a translate of a tile is a tile, so we may assume $e\in T$.
We will say a group $G$ is {\em monotileably amenable} if for every finite set $K\subseteq G$ and every $\eps>0$,
there is a tile $T$ for $G$ that is $(K,\eps)$--invariant.
This notion was introduced (though not named with quite the same words we use here)
by B.\ Weiss in his paper~\cite{W01}, where he proved that every 
residually finite amenable group and every solvable group is monotileably amenable.
This class of groups includes, in addition to the solvable groups, all linear amenable groups and
Grigorchuk's groups~\cite{Gri.84} of intermediate growth.
It is an open problem whether all amenable groups are monotileably amenable, and this is not even known for the elementary
amenable groups.
However,
as shown by Ornstein and Weiss~\cite{OW87},
all amenable groups do admit quasitilings,
involving finite sets of quasitiles and approximations, and this circle of ideas,
as further developed by Kerr and Li~\cite{KL},
plays an important role in our proof.

All sofic groups are hyperlinear.
An application of results of~\cite{BDJ08} is that the class of hyperlinear
groups is closed under taking free products with amalgamation over amenable subgroups,
and this result inspired our effort in this paper.
The techniques of~\cite{BDJ08} do not appear adapted to prove that a group is sofic.
The proof in~\cite{BDJ08} relied on approximation of group von Neumann algebras of amenable groups
by finite dimensional algebras, which is not helpful in the context of this paper.
However, one aspect of the proof found here is reminiscent of the proof in~\cite{BDJ08}:  the use
of independent random unitaries to model freeness with amalgamation.
In~\cite{BDJ08}, the random unitaries were distributed according to Haar measure in the group of unitary matrices
that commute with a certain finite dimensional subalgebra,
whereas here we use uniformly distributed random permutation matrices.
See Remark~\ref{rem:cheap} for more about this.

To be more precise, our construction of quasi--actions of amalgamated free product groups $\Gamma_1*_H\Gamma_2$
where $H$ is monotileably amenable
goes by proving asymptotic vanishing of certain moments involving random permutation matrices.

Asymptotic freeness of indpendent matrices (of various sorts) as the matrix
size grows without bound is one of the mainstays of free probability theory, going back to seminal work~\cite{V91}
of Voiculescu,
and has been a key element in applications of free probability theory to operator algebras and elsewhere.
Asymptotic freeness of independent random permutation matrices was proved by A.\ Nica~\cite{N93}.
By combining Nica's result with our vanishing of moments result, we are able to extend 
Nica's asymptotic freeness result to the case of independent random permutation matrices {\em and}
certain sequences of non--random permutation matrices;
these are essentially sequences that arise from quasi--actions of sofic groups.

\medskip

The organization of the rest of this paper is as follows:
in Section~\ref{sec:vanish}, we prove our main technical result on asymptotic vanishing
of certain moments in random permutation matrices and certain non--random matrices;
in Section~\ref{sec:sofic}, we apply this asymptotic vanishing theorem to prove our main result,
that the class of sofic groups is closed under taking free products with amalgamation over monotileably amenable subgroups;
in Section~\ref{sec:free}, we combine the result of Section~\ref{sec:vanish}
with Nica's asymptotic freeness result and extend Nica's result to handle certain non--random permutation
matrices too.

\bigskip
\noindent{\em Acknowledgement.}  The authors thank Ion Nechita for the proof of Lemma 2.2,
which is an improvement on their first version.
They also thank Alexey Muranov for pointing out an error in an earlier version of this paper,
and Lewis Bowen and David Kerr for helpful discussions about quasitilings of amenable groups.

\section{Asymptotic vanishing of certain moments}
\label{sec:vanish}

The main result of this section (Theorem~\ref{thm:BU}) is an asymptotic vanishing of moments result
involving uniformly distributed random permutation matrices
and (sequences of) non--random permutation matrices whose traces approach zero as matrix size increases.
Actually, a broader class than permutation matrices is considered here, which is needed for applications.
The theorem is used in the next section to prove the main result of the paper.

We begin by fixing some notation and definitions.
If $Z$ is a finite set, then a {\em partition} of $Z$ is a set $p=\{X_1,\ldots,X_n\}$
of pairwise disjoint, nonempty subsets $X_j$ of $Z$ whose union is all of $Z$.
These sets $X_j$ are called the {\em blocks} of the partition, and the number of blocks of $p$ is denoted simply $|p|$.
We then have the equivalence relation $\simp$ on $Z$ defined by $z_1\simp z_2$
if and only if $z_1$ and $z_2$ belong to the same block of $p$.

If $Y\subset Z$ is a nonempty subset, then we let $p\restrict_Y$ denote the restriction of $p$ to Y, namely
\[
p\restrict_Y=\{X\cap Y\mid X\in p,\,X\cap Y\ne\emptyset\}.
\]

We let $\Pc(n)$ denote the set of all partitions of $\{1,\ldots,n\}$ and let $\le$ be the usual ordering of $\Pc(n)$
given by $r\le s$ if and only if every block of $r$ is contained in some block of $s$.
This makes $\Pc(n)$ into a lattice, and we use $\vee$ and $\wedge$ for the join and meet operations in this lattice.

If $i=(i_1,\ldots ,i_n)$ be a multi index with values in $\{1,\ldots ,d\}$ and $p\in \Pc(n)$, then we define
\begin{equation}\label{eq:deltaip}
\delta_{i,p}=\begin{cases}1,&\text{if }k\simp\ell\text{ implies }i_k=i_\ell\\0,&\text{otherwise.}\end{cases}
\end{equation}

Let $U$ be a random $d\times d$ permutation matrix that is uniformly distributed
and let us write $U=(u_{i_1,i_2})_{1\le i_1,i_2\le d}$,
keeping in mind the dependence of everything on $d$.
We let $\Tr$ denote the usual trace on complex matrix algebras
(normalized so that projections of rank $1$ have trace $1$) and $\tr_d=\frac1d\Tr:M_d(\Cpx)\to\Cpx$.

Suppose for every $j,d\in\Nats$, $B^{(d)}_j$ is a $d\times d$
matrix,
all of whose entries are $0$ and $1$, with each row and each column having at most one nonzero entry.
For example, $B^{(d)}_j$ could be permutation matrices.
We will write $B^{(d)}_j=(b^{(j,d)}_{i_1,i_2})_{1\le i_1,i_2\le d}$ and often simply
$B^{(d)}_j=B_j=(b^{(j)}_{i_1,i_2})_{1\le i_1,i_2\le d}$, keeping in mind the dependence on $d$.

\begin{thm}\label{thm:BU}
With $B_1,\ldots,B_{2n}$ and $U$ as above,
there are constants $C_n$ and $D_n$ depending only on $n$ such that,
letting
\begin{equation}\label{eq:fd}
f(d)=\max_{1\le j\le 2n}\tr_d(B_j),
\end{equation}
we have
\begin{equation}\label{eq:trBUbnd}
\int\bigg(\tr_d\big(B_1(UB_2U^*)B_3(UB_4U^*)\cdots B_{2n-1}(UB_{2n}U^*)\big)\bigg)\,dU \\
\le C_nf(d)+D_nd^{-1}.
\end{equation}
Thus,
if
$\lim_{d\to\infty}\tr_d(B^{(d)}_j)=0$
for all $j$, then we have
\[
\lim_{d\to\infty}\int\bigg(\tr_d\big(B_1(UB_2U^*)B_3(UB_4U^*)\cdots B_{2n-1}(UB_{2n}U^*)\big)\bigg)\,dU=0.
\]
\end{thm}
\begin{proof}
We have
\begin{multline*}
\int\bigg(\tr_d\big(B_1(UB_2U^*)B_3(UB_4U^*)\cdots B_{2n-1}(UB_{2n}U^*)\big)\bigg)\,dU \\
=\frac1d\sum_{1\le i_1,\ldots,i_{4n}\le d}
\begin{aligned}[t]
b_{i_1,i_2}^{(1)}&b_{i_3,i_4}^{(2)}\cdots b_{i_{4n-1},i_{4n}}^{(2n)} \\
&\cdot\int u_{i_2,i_3}u_{i_5,i_4}u_{i_6,i_7}u_{i_9,i_8}\cdots u_{i_{4n-2},i_{4n-1}}u_{i_1,i_{4n}}\,dU.
\end{aligned}
\end{multline*}
Moreover, as is easily verified, for $k_1,\ldots,k_m,\ell_1,\ldots,\ell_m\in\{1,\ldots,d\}$, we have
\[
\int u_{k_1,\ell_1}u_{k_2,\ell_2}\cdots u_{k_m,\ell_m}\,dU=
\begin{cases}
\frac{(d-|r|)!}{d!},&r=s \\
0,&r\ne s,
\end{cases}
\]
where $r$ and $s$ are the partitions of $\{1,\ldots,m\}$ defined by $i\overset{r}{\sim}j$ if and only if $k_i=k_j$
and  $i\overset{s}{\sim}j$ if and only if $\ell_i=\ell_j$.
Therefore, we have
\begin{multline}\label{eq:Irsum}
\left|\int\bigg(\tr_d\big(B_1(UB_2U^*)B_3(UB_4U^*)\cdots B_{2n-1}(UB_{2n}U^*)\big)\bigg)\,dU\right| \\
\le\frac1d\sum_{r\in \Pc(2n)}\frac{(d-|r|)!}{d!}\sum_{i\in I(r)}b_{i_1,i_2}^{(1)}b_{i_3,i_4}^{(2)}\cdots b_{i_{4n-1},i_{4n}}^{(2n)}\,,
\end{multline}
where $I(r)=I(r,d)$ is the set of all $i=(i_1,\ldots,i_{4n})\in\{1,\ldots,d\}^{4n}$ such that $i_a=i_b$ whenever $a\simp b$,
where $p=p(r)\in\Pc(4n)$ is the partition that is the union of $f$ applied to $r$ and $g$ applied to $r$, where
$f,g:\{1,\ldots,2n\}\to\{1,\ldots,4n\}$ are given by
\begin{align*}
f(j)&=\begin{cases}
2j,&j\text{ odd}, \\
2j+1,&j\text{ even and }j<2n \\
1,&j=2n,
\end{cases} \\[1ex]
g(j)&=\begin{cases}
2j+1,&j\text{ odd}, \\
2j,&j\text{ even.}
\end{cases}
\end{align*}
These functions are presented in Table~\ref{tab:rps}.
\begin{table}[ht]
\caption{The functions $f$ and $g$, used to form the partition $p$ from $r$.}
\label{tab:rps}
\begin{tabular}{r|c|c|c|c|c|c|c|c|c|c|c|c|c}
$f$ maps & $2n$ & $1$ & & & $2$ & $3$ & & & $4$ & $5$ & & & $\cdots$ \\ \hline
$g$ maps &  & & $1$ & $2$ & & & $3$ & $4$ & & & $5$ & $6$ & $\cdots$  \\ \hline
to & $1$ & $2$ & $3$ & $4$ & $5$ & $6$ & $7$ & $8$ & $9$ & $10$ & $11$ & $12$ & $\cdots$
\end{tabular} \hfill

\vskip1ex
\begin{tabular}{r|c|c|c|c|c}
$f$ maps & \hspace*{7em} $\cdots$ & $2n-2$ & $2n-1$ & & \\ \hline
$g$ maps & \hspace*{7em} $\cdots$ & & & $2n-1$ & $2n$ \\ \hline
to &  \hspace*{7em} $\cdots$ & $4n-3$ & $4n-2$ & $4n-1$ & $4n$
\end{tabular}
\end{table}

An upper bound for the right--hand--side of~\eqref{eq:Irsum} when $d\ge4n$ is
\begin{equation}\label{eq:deltaipsum}
2^{2n}\sum_{r\in\Pc(2n)}d^{-|r|-1}\sum_{1\le i_1,\ldots,i_{4n}\le d}\delta_{i,p(r)}b_{i_1,i_2}^{(1)}b_{i_3,i_4}^{(2)}\cdots b_{i_{4n-1},i_{4n}}^{(2n)}\,,
\end{equation}
where $\delta_{i,p(r)}$ is as defined in~\eqref{eq:deltaip}.

We will need the following result, which is purely about partitions:
\begin{lemma}\label{lem:oddevnbs}
Let $n\ge1$ and suppose $r\in\Pc(2n)$ satisfies
$2j-1\overset{r}{\not\sim}2j$ for all $j\in\{1,\ldots,n\}$.
Let $\eta=\{\{1,2\},\{3,4\},\ldots,\{2n-1,2n\}\}\in\Pc(2n)$.
Then $|r\vee\eta|\le|r|/2$.
\end{lemma}
\begin{proof}
Each block $X$ of $r\vee\eta$ contains at least two blocks of $r$, because if $X$
were equal to a block of $r$, then it would also be a union of blocks of $\eta$, which is impossible by
the hypothesis on $r$.
This finishes the proof of Lemma~\ref{lem:oddevnbs}.
\end{proof}

The following lemma will be used to handle the right--most sum in~\eqref{eq:deltaipsum}.
\begin{lemma}\label{lem:bi}
Let $n\in\Nats$ and let $p$ be a partition of $\{1,2,\ldots,2n\}$.
Consider $(0,1)$--matrices $B_j$ having at most one nonzero entry per row and column, (as in Theorem~\ref{thm:BU}).
Let
\begin{equation}\label{eq:Spd}
S(p,d)=S(B_1,\ldots,B_n;p,d):=\sum_{1\le i_1,\ldots,i_{2n}\le d}\delta_{i,p}\;b_{i_1,i_2}^{(1)}b_{i_3,i_4}^{(2)}\cdots b_{i_{2n-1},i_{2n}}^{(n)}\,.
\end{equation}
Consider $p\vee\eta$ where $\eta=\{\{1,2\},\{3,4\},\ldots,\{2n-1,2n\}\}\in\Pc(2n)$.
Then $S(p,d)\le d^{|p\vee\eta|}$.
Moreover, if
\begin{equation}\label{eq:2kcond}
2j-1\simp2j\text{ for some }j\in\{1,\ldots,n\},
\end{equation}
then letting
$f(d)=\max_{1\le j\le n}\tr_d(B_j)$,
we have $S(p,d)\le f(d)d^{|p\vee\eta|}$.
\end{lemma}
\begin{proof}
Writting $p\vee\eta=\{X_1,\ldots,X_m\}$, we have that $p$ is the disjoint union $p_1\cup\cdots\cup p_m$, where $p_k$ is a partition
of $X_k$.
Then
\[
S(B_1,\ldots,B_n;p,d)=\prod_{k=1}^mS(B_{i(k,1)},B_{i(k,2)},\ldots,B_{i(k,\ell_k)};\pt_k,d),
\]
where $X_k=\{i(k,1),\ldots,i(k,\ell_k)\}$ for $i(k,1)<i(k,2)<\cdots<i(k,\ell_k)$
and where $\pt_k$ is the appropriate renumbering of $p_k$.
Since the condition~\eqref{eq:2kcond} holds for $p$ if and only if it holds for some $p_k$,
and since $f(d)\le1$ for all $d$,
it will suffice to prove the lemma in the case that
$p\vee\eta$ has only one block.

Suppose $p\vee\eta$ has only one block.
Fix $i_1,\ldots,i_{2n}\in\{1,\ldots,d\}$
and suppose we have
\begin{equation}\label{eq:bnon0}
\delta_{i,p}\;b_{i_1,i_2}^{(1)}b_{i_3,i_4}^{(2)}\cdots b_{i_{2n-1},i_{2n}}^{(n)}\ne0.
\end{equation}
Since each row and column of each $B_j$ has at most one nonzero entry,
for any given $k\in\{1,\ldots,d\}$ there is at most one value of $k'\in\{1,\ldots,d\}$ such that $b_{k,k'}^{(j)}\ne0$.
Since $p\vee\eta$ has only one block, for any given $i_1\in\{1,\ldots,d\}$ there is at most one choice of $i_2,\ldots,i_{2n}$
such that~\eqref{eq:bnon0} holds.
This implies $S(p,d)\le d$, as required.

Now suppose $p\vee\eta$ (still) has only one block and $2j-1\simp 2j$ for some $j\in\{1,\ldots,n\}$.
For any choice of $i_{2j}\in\{1,\ldots,d\}$, there is at most one choice of $i_1,\ldots,i_{2j-1},i_{2j+1},\ldots,i_{2n}$
such that~\eqref{eq:bnon0} holds; in this choice, we must have $i_{2j-1}=i_{2j}$, because $\delta_{i,p}\ne0$.
Therefore, 
\[
0\le S(p,d)\le\sum_{i_{2j}=1}^db_{i_{2j},i_{2j}}^{(j)}=\Tr(B_j)=\tr_d(B_j)d\le f(d)d
\]
This finishes the proof of Lemma~\ref{lem:bi}.
\end{proof}

\medskip

\noindent{\em Completion of the proof of Theorem~\ref{thm:BU}.}
Consider $p\vee\gamma$, where
\[
\gamma=\{\{1,2\},\{3,4\},\ldots,\{4n-1,4n\}\}\in\Pc(4n).
\]
Then $|p\vee\gamma|=|r\vee\eta|+|r\vee\eta'|$, where
$\eta=\{\{1,2\},\{3,4\},\ldots,\{2n-1,2n\}\}$ and
$\eta'=\{\{2n,1\},\{2,3\},\{4,5\},\ldots,\{2n-2,2n-1\}\}$ are from $\Pc(2n)$.
By Lemma~\ref{lem:bi},
we have
\[
S(B_1,\ldots,B_{2n};p,d)\le d^{|r\vee\eta|+|r\vee\eta'|}.
\]
Furthermore, 
if $2j-1\simp2j$ for some $j\in\{1,\ldots,2n\}$, then we have
\[
S(B_1,\ldots,B_{2n};p,d)\le f(d)\,d^{|r\vee\eta|+|r\vee\eta'|},
\]
where $f(d)$ as in~\eqref{eq:fd}.

Therefore, using the upper bound~\eqref{eq:deltaipsum} for
the integral in~\eqref{eq:Irsum},
in order to finish the proof of~\eqref{eq:trBUbnd}, it will suffice to prove:
for any $r\in\Pc(2n)$, we have
\begin{equation}\label{eq:rs1}
|r\vee\eta'|+|r\vee\eta|\le|r|+1
\end{equation}
while if, furthermore,
\begin{equation}\label{eq:rnopair}
j-1\overset{r}{\not\sim}j,\qquad(j\in\{1,\ldots,2n-1\})\qquad\text{ and }1\overset{r}{\not\sim}2n,
\end{equation}
then we have
\begin{equation}\label{eq:rs0}
|r\vee\eta'|+|r\vee\eta|\le|r|.
\end{equation}

Let us first show that~\eqref{eq:rs1} holds for all $r\in\Pc(2n)$.
We write $r=\{X_1\ldots,X_m\}$ for some nonempty sets $X_j$ and $m\ge1$.
If $m=1$, then~\eqref{eq:rs1} holds,
because, we have $|r\vee\eta'|\le|r|$ and $|s\vee\eta|\le|s|$.
Suppose $m\ge2$.
We claim that there are $a_1,\ldots,a_{m-1},b_1,\ldots,b_{m-1}\in\{1,\ldots,m\}$ with
\begin{alignat*}{2}
b_i&\not\in\{a_1\}\cup\{b_1,\ldots,b_{i-1}\}&&\quad(1\le i\le m-1) \\
a_i&\in\{a_1\}\cup\{b_1,\ldots,b_{i-1}\}&&\quad(2\le i\le m-1)
\end{alignat*}
and with some $j_i\in X_{a_i}$, $k_i\in X_{b_i}$ such that either $j_i\overset{\eta}{\sim}k_i$
or $j_i\overset{\eta'}{\sim}k_i$, (i.e., such that $j_i$ and $k_i$ are distance $1$ apart, modulo $2n$).
Indeed if $Y$ is any proper, nonempty subset of $\{1,\ldots,2n\}$, then 
the complement of $Y$ must contain some element that is distance $1$ from some element of $Y$.
So to prove the claim about the $a_i$ and $b_i$, we start with $a_1=1$;
taking $Y=X_{a_1}$, what we just showed implies that there is $b_1\ne1$ and $j_1\in X_{a_1}$, $k_1\in X_{b_1}$
so that either $j_1\overset{\eta}{\sim}k_1$ or $j_1\overset{\eta'}{\sim}k_1$.
If $m=2$, then we are done.
Otherwise, letting $Y=X_{a_1}\cup X_{b_1}$, what we showed implies that there are
$a_2\in\{a_1,b_1\}$ and $b_2\in\{1,\ldots,m\}\backslash\{a_1,b_1\}$
with $j_2\in X_{a_2}$ and $k_2\in X_{b_2}$ such that either
$j_1\overset{\eta}{\sim}k_1$ or $j_1\overset{\eta'}{\sim}k_1$.
If $m=3$, then we are done.
Otherwise, we continue in this manner, letting $Y=X_{a_1}\cup X_{b_1}\cup X_{b_2}$ and finding $a_3\in\{a_1,b_1,b_2\}$
and $b_3\not\in\{a_1,b_1,b_2\}$ and $j_3\in X_{a_3}$, $k_3\in X_{b_3}$ as required.
We continue until we have selected $m-1$ such pairs.
This proves the claim.

We may form each of $r\vee\eta$ and $r\vee\eta'$ from $r$ by performing identifications one at a time.
Thus, we construct successively a sequence
$r=r_0\le r_1\le\cdots\le r_\ell=r\vee\eta$ such that $|r_i|=|r_{i-1}|-1$ and $r_i$ is obtained from $r_{i-1}$
by merging two distinct blocks of $r_{i-1}$ that contain, respectively, $j$ and $k$ with $j\overset{\eta}{\sim}k$,
and we similarly construct a sequence
$r=r'_0\le r'_1\le\cdots\le r'_{\ell'}=r\vee\eta'$
from $r$ to $r\vee\eta'$, by performing identifications implied by $\overset{\eta'}{\sim}$.
The choice of $j_i\in X_{a_i}$ and $k_i\in X_{b_i}$ with $b_i\not\in\{a_1\}\cup\{b_1,\ldots,b_{i-1}\}$
found in the previous claim
shows that, by choosing the identifications 
accordingly, when forming $r\vee\eta$ and $r\vee\eta'$ from $r$, an aggregate of at least $m-1$ such identifications
is made.
Therefore, we have $2|r|-|r\vee\eta|-|r\vee\eta'|\ge m-1$.
But we have $|r|=m$,
which yields~\eqref{eq:rs1}.

Now we prove the inequality~\eqref{eq:rs0} under the additional hypothesis~\eqref{eq:rnopair}.
By applying Lemma~\ref{lem:oddevnbs} to $r$ and to a rotation of $r$, we get $|r\vee\eta|\le|r|/2$ and $|r\vee\eta'|\le|r|/2$.
This gives immediately~\eqref{eq:rs0},
and finishes the proof of Theorem~\ref{thm:BU}.
\end{proof}

\section{Sofic groups}
\label{sec:sofic}

In this section, we apply Theorem~\ref{thm:BU} to prove our main result.

\begin{lemma}\label{lem:tiledFolner}
Let $T$ be a tile for a countable amenable group $G$, with $e\in T$,
and let $C$ be a set of centers.
Let $K\subseteq G$ be a finite subset and $\eps>0$.
Then there is a $(K,\eps)$--invariant set $F$
of the form $F=TD$ for $D\subseteq C$.
\end{lemma}
\begin{proof}
We may choose a finite subset $E\subseteq G$ that is $(T\,T^{-1},\frac\eps{2|K|})$--invariant
and also $(K,\frac\eps2)$--invariant.
Let $D=(T^{-1}E)\cap C$ and let $F=TD$.
Then $E\subseteq F\subseteq T\,T^{-1}E$,
and
\[
KF\backslash F\subseteq (KF\backslash KE)\cup(KE\backslash E)
\subseteq K(F\backslash E)\cup(KE\backslash E).
\]
Therefore,
\[
\frac{|KF\backslash F|}{|F|}\le\frac{|K||F\backslash E|+|KE\backslash E|}{|E|}
\le \frac{|K|\,|T\,T^{-1}E\backslash E|+|KE\backslash E|}{|E|}<\eps.
\]
\end{proof}

For a set $X$, by $\Sym(X)$ we denote the set of all permutations of $X$.
Thus, we have $S_d=\Sym(\{1,\ldots,d\})$.
For maps $\phi:G\to\Sym(X)$, when $G$ is a group,
we will frequently write $\phi_g$ instead of $\phi(g)$.

\begin{remark}\label{rem:closesets}
Let $0<\delta<1$ and let $X$ and $X'$ be nonempty finite sets.
If
$Y\subseteq X$ and $Y'\subseteq X'$ are finite subsets satisfying
$|Y|>(1-\delta)|X|$ and $|Y'|>(1-\delta)|X'|$
and if $\alpha:Y\to Y'$ is a bijection, 
then for $G$ any group and $\phi:G\to\Sym(X)$ any map, we can define $\phi':G\to\Sym(X')$
by letting $\phi'_g\circ\alpha(x)=\alpha\circ\phi_g(x)$ whenever $x\in Y\cap\phi_g^{-1}(Y)$,
which defines $\phi'_g$ on all but at most $\delta(|X'|+|X|)\le2\frac{\delta}{1-\delta}|X'|$ of the points of $X'$,
and by choosing some values for $\phi'_g$ on the other points in order to make it a permutation.
If $F$ is a finite subset of $G$ and if $\phi$ is an $(F,\eps)$--quasi--action, then it follows that $\phi'$
is an $(F,\eta)$--quasi--action of $G$ on $X'$, where $\eta=\eps+6\delta/(1-\delta)$.
\end{remark}

Lemma~4.5 of~\cite{KL} could be described as yielding quasitilings for quasi--actions of amenable groups.
The following is an application of it in the case that the group has a tile.
In effect, we tile each of the quasitiles with our fixed monotile.
\begin{lemma}\label{lem:TZ}
Let $G$ be an amenable group and suppose
$T$ is a tile of $G$, with $e\in T$.
Then for every $\delta>0$, there is $\delta'>0$ and a finite set $F\subseteq G$, with $TT^{-1}\subseteq F$, such that
if $\phi:G\to\Sym(X)$ is an $(F,\delta')$--quasi--action of $G$, then there is a set $Z$ and there
are subsets $Y\subseteq X$ and $Y'\subseteq T\times Z$ with $|Y|>(1-\delta)|X|$ and $|Y'|>(1-\delta)|T\times Z|$
and there is a bijection $\alpha:Y\to Y'$ such that
\[
\alpha\circ\phi_g\circ\alpha^{-1}(t,z)=(gt,z)
\]
whenever $(t,z)\in Y'$, $g\in G$, $gt\in T$ and $\phi_g\circ\alpha^{-1}(t,z)\in Y$.
\end{lemma}
\begin{proof}
Let $C\subseteq G$ be a set of centers for the tile $T$.
Using Lemma~\ref{lem:tiledFolner}, for any $\ell\in\Nats$ and $\eta'>0$
we can find sets $F_1\subseteq F_2\subseteq\cdots\subseteq F_\ell$
of the form $F_k=TD_k$ for some nonempty subsets $D_k\subseteq C$, and with
$T\,T^{-1}\subseteq F_1$ and $|F_{k-1}^{-1}F_k\backslash F_k|<\eta'|F_k|$
for $k\in\{2,3,\ldots,\ell\}$.

We now apply Lemma~4.5 of~\cite{KL} with $\tau=0$ and with some $\eta>0$ to be specified later.
This lemma and its proof imply that
there exist $\ell\in\Nats$ and $\eta',\eta''>0$ such that whenever $F_1\subseteq F_2\subseteq\cdots\subseteq F_\ell$
are chosen as above, then for the finite set $F=F_\ell\cup F_\ell^{-1}\subseteq G$,
if $X$ is a finite set and if $\phi:G\to\Sym(X)$
is a map and if $B\subseteq X$ satisfies
\begin{enumerate}[(i')]
\item $|B|\ge(1-\eta'')|X|$
\item $\phi_{st}(a)=\phi_s\phi_t(a)$, $\phi_s(a)\ne\phi_{s'}(a)$ and $\phi_e(a)=a$
      for all $a\in B$ and all $s,s',t\in F$ with $s\ne s'$,
\end{enumerate}
then there exist sets $C_1,\ldots, C_\ell\subseteq X$ such that
\begin{enumerate}[(i)]
\item for all $k\in\{1,\ldots,\ell\}$ the map $F_k\ni s\mapsto\phi_s(c)$ is injective
\item the sets $\phi(F_1)C_1,\ldots,\phi(F_\ell)C_\ell$ are pairwise disjoint and
the sets
\[
(\phi(F_k)c)_{1\le k\le\ell,\,c\in C_k}
\]
are $\eta$--disjoint and $(1-\eta)$--cover $X$.
\end{enumerate}

We will choose $\delta'$ so small that $\phi:G\to\Sym(X)$ being an $(F,\delta')$--quasi--action will ensure
the existence of $B$ such that the hypotheses
(i') and (ii') hold.
Then we let $X'$ be the disjoint union $\coprod_{k=1}^\ell F_k\times C_k$.
We can find subsets $Y'\subseteq X'$ and $Y\subseteq X$ such that $|Y'|\ge(1-\eta)|X'|$
and $|Y|\ge(1-2\eta)|X|$ and a bijection $\alpha:Y\to Y'$ such that
whenever $s\in F$, $(t,c)\in Y'\cap(F_k\times C_k)$ and $st\in F_k$, we have
$\alpha\circ\phi_s\circ\alpha^{-1}(t,c)=(st,c)$.
Since $F_k=TD_k$, we have a natural identification of $X'$ with $T\times Z$, where $Z$ is the disjoint union
$\coprod_{k=1}^\ell D_k\times C_k$.
Since $T\subseteq F_1$ and $TT^{-1}\subseteq F$, by chosing $\eta=\delta/2$, we are done.
\end{proof}

\begin{thm}\label{thm:overamen}
Let $\Gamma=\Gamma_1*_H\Gamma_2$ be a free product of groups with amalgamation over
a subgroup $H$.
Assume that $\Gamma_1$ and $\Gamma_2$ are sofic and that $H$ is a monotileably amenable group.
Then $\Gamma$ is sofic.
\end{thm}
\begin{proof}
We may without loss of generality assume that $\Gamma_1$ and $\Gamma_2$ are countable.
Let $R_{i,1}\subseteq R_{i,2}\subseteq\cdots$ be finite subsets of $\Gamma_i$ whose union is all of $\Gamma_i$.
Let $K_p=(R_{1,p}\cup R_{2,p})\cap H$ and
let $T_p$ be a tile for $H$ such that 
\begin{equation}\label{eq:KT}
|(K_pT_p)\backslash T_p|<\frac1p|T_p|.
\end{equation}
Fix a map $\rho_p:K_p\to\Sym(T_p)$ so that $(\rho_p)_h(t)=ht$ whenever $t\in T_p$, $h\in K_p$ and $ht\in T_p$.

We fix some sequence $\delta_p$ tending to $0$, to be specified later.
Now applying Lemma~\ref{lem:TZ} in the case of $T_p\subseteq H$ and $\delta_p$,
we find finite sets $F_p\subseteq H$ with $T_pT_p^{-1}\subseteq F_p$ and we find $\delta_p'>0$
as described there.
We assume (without loss of generality) $\delta_p'<\delta_p$.
In particular, letting $\phi_{i,p}:\Gamma_i\to\Sym(X_{i,p})$ be an $(F_p\cup R_{i,p},\delta_p')$--quasi--action of $\Gamma_i$,
we find sets $Z_{i,p}$ and subsets $Y_{i,p}\subseteq X_{i,p}$ and $Y_{i,p}'\subseteq T_p\times Z_{i,p}$
with $|Y_{i,p}|>(1-\delta_p)|X_{i,p}|$ and $|Y_{i,p}'|>(1-\delta_p)|T_p\times Z_{i,p}|$ and bijections
$\alpha_{i,p}:Y_{i,p}\to Y_{i,p}'$ such that
\[
\alpha_{i,p}\circ(\phi_{i,p})_h\circ\alpha_{i,p}^{-1}(t,z)=(ht,z)
\]
whenever $(t,z)\in Y_{i,p}'$, $h\in H$, $ht\in T$ and $(\phi_{i,p})_h\circ\alpha_{i,p}^{-1}(t,z)\in Y_{i,p}$.
As described in Remark~\ref{rem:closesets}, we thus obtain
$(F_p\cup R_{i,p},\eta_p)$--quasi--actions
$\phi_{i,p}:\Gamma_i\to\Sym(T_p\times Z_{i,p})$
such that
\begin{equation}\label{eq:phi'}
(\phi_{i,p}')_h(t,z)=(ht,z)
\end{equation}
whenever $z\in Z_{i,p}$, $h\in H$ and $ht\in T_p$,
where $\eta_p=\delta_p'+6\delta_p/(1-\delta_p)$.
By amplification, if necessary, we may without loss of generality assume $Z_{1,p}=Z_{2,p}$, and we denote this set
by $Z_p$.
We may further amplify, if necessary, in order to make the cardinality of $Z_p$ as large as desired.

Thinking of elements of $\Sym(T_p\times Z_p)$ as permutation matrices and elements of $M_{|T_p|\,|Z_p|}(\Cpx)$,
making the obvious identification of this matrix algebra with $M_{|T_p|}(\Cpx)\otimes M_{|Z_p|}(\Cpx)$
and letting $(e_{t,t'})_{t,t'\in T_p}$ be the usual system of matrix units for $M_{|T_p|}(\Cpx)$,
we have for each $g\in\Gamma_i$
\begin{equation}\label{eq:Bg}
(\phi_{i,p}')_g=\sum_{t,t'\in T_p}e_{t,t'}\otimes B^{(i)}_{g,t,t'}\in M_{|T_p|}(\Cpx)\otimes M_{|Z_p|}(\Cpx),
\end{equation}
where each $B^{(i)}_{g,t,t'}$ is a $(0,1)$--matrix having at most one $1$ in each row and column.
Fixing any $t,t'\in T_p$ and letting $h=t(t')^{-1}$, from~\eqref{eq:phi'} we see that $B^{(i)}_{h,t,t'}$ is the
identity matrix.
Using that $\phi_{i,p}'$ is an $(F_p\cup R_{i,p},\eta_p)$--quasi--action and that $T_pT_p^{-1}\subseteq F_p$,
we see that for every $g\in R_{i,p}\backslash H$,
the permutation $(\phi_{i,p}')_g(\phi_{i,p}')_{t(t')^{-1}}^{-1}$ has at most $2\eta_p|T_p\times Z_p|$ fixed points;
this implies that $B^{(i)}_{g,t,t'}$ has at most $2\eta_p|T_p\times Z_p|$ diagonal entries that are equal to $1$.
In other words, for $g\in R_{i,p}\backslash H$ and all $t,t'\in T_p$, we have
\begin{equation}\label{eq:trBg}
\tr_{|Z_p|}(B^{(i)}_{g,t,t'})\le2\eta_p|T_p|.
\end{equation}

Let $U_p$ be a uniformly distributed random $|Z_p|\times |Z_p|$ permutation matrix, and let $V_p=1\otimes U_p$, taking values
in $M_{|T_p|}(\Cpx)\otimes M_{|Z_p|}(\Cpx)$.
Take $n\in\Nats$ and take
\begin{equation*}
g_j\in\begin{cases}
R_{1,p}\backslash H,&j\text{ odd} \\[1ex]
R_{2,p}\backslash H,&j\text{ even}
\end{cases}
\end{equation*}
and consider the moment
\begin{equation}\label{eq:Vmom}
\tr_{|T_p\times Z_p|}\big(\,\phi_{1,p}'(g_1)\big(V_p\phi_{2,p}'(g_2)V_p^*\big)\cdots\phi_{1,p}'(g_{2n-1})\big(V_p\phi_{2,p}'(g_{2n})V_p^*\big)\,\big),
\end{equation}
thought of as a random variable.
Writing out $V_p=1\otimes U_p$ and using~\eqref{eq:Bg}, we find that the moment~\eqref{eq:Vmom} equals the sum
\begin{equation*}
\frac1{|T_p|}\sum_{t_1,t_2,\ldots,t_{2n}\in T_p}
\tr_{|Z_p|}\begin{aligned}[t]
  \bigg(B_{g_1,t_1,t_2}^{(1)}&\big(U_pB_{g_2,t_2,t_3}^{(2)}U_p^*\big)\cdots \\
      &B_{g_{2n-1},t_{2n-1},t_{2n}}^{(1)}\big(U_pB_{g_{2n},t_{2n},t_1}^{(2)}U_p^*\big)\bigg).
      \end{aligned}
\end{equation*}
Using Theorem~\ref{thm:BU} and~\eqref{eq:trBg}, we find an upper bound for the expectation of the above sum to be
\begin{equation}\label{eq:ubnd}
|T_p|^{2n-1}\big(C_n(2\eta_p|T_p|)+\frac{D_n}{|Z_p|}\big),
\end{equation}
where $C_n$ and $D_n$ are the constants from Theorem~\ref{thm:BU}.
Since
$\eta_p\le\delta_p+6\delta_p/(1-\delta_p)$
can be made arbitrarily small
by choosing $\delta_p$ small enough, and since $|Z_p|$ can be made as large as needed, we choose $\delta_p$ and $|Z_p|$
so that for every $n$, the upper bound~\eqref{eq:ubnd} tends to zero as $p\to\infty$.

Now we modify $\phi_{i,p}'$ on $K_p$ so that they agree for $i=1,2$.
By the estimate~\eqref{eq:KT} and the formula~\eqref{eq:phi'},
letting
\[
(\phi_{i,p}'')_g=
\begin{cases}
(\rho_p)_g\times\id_{Z_p}\,,&g\in K_p \\
(\phi'_{i,p}(g))_g\,,&\text{otherwise,}
\end{cases}
\]
we see that $\phi_{i,p}''$ is an $(R_{i,p}\cup F_p,\eta_p+\frac6p)$--quasi--action of $\Gamma_i$.
Moreover, since $(\phi_{i,p}'')_g$ agrees with $(\phi_{i,p}')_g$ if $g\notin H$, the moment~\eqref{eq:Vmom}
still tends to zero as $p\to\infty$ when $\phi_{i,p}''$ replaces $\phi_{i,p}'$.
Note that $V_p$ commutes with $\phi_{i,p}''(h)$ for all $h\in K_p$.

Now we will change our random permutation matrix $V_p$ to a non--random permutation matrix, at the cost of increasing
the matrix size.
Indeed, $V_p$ takes on $|Z_p|!$ different values in $M_{|T_p|\,|Z_p|}(\Cpx)$, each with equal probability.
So define $\phit_{i,p}:\Gamma_i\to M_{|T_p|\,|Z_p|(|Z_p|!)}(\Cpx)$ by letting $\phit_{i,p}(g)$ be the block diagonal permutation matrix consisting
of $|Z_p|!$ copies of $\phi_{i,p}''(g)$ down the diagonal, and let $\Vt_p$ be the block diagonal permutation matrix 
consisting of the $|Z_p|!$ different values taken by $V_p$ repeated one after the other down the diagonal.
Now it is clear that the expectation of the trace $\tr_{|T_p|\,|Z_p|}$ applied to a word with
letters taken from $\phi_{1,p}''(\Gamma_1)$, $\phi_{2,p}''(\Gamma_2)$ and $\{V_p,V_p^*\}$
equals the trace $\tr_{|T_p|\,|Z_p|(|Z_p|!)}$ applied to the corresponding word of letters taken from
$\phit_{1,p}(\Gamma_1)$, $\phit_{2,p}(\Gamma_2)$ and $\{\Vt_p,\Vt_p^*\}$.
Upon identifying permutation matrices with permutations, we have that $\phit_{i,p}$ is an
$(R_{i,p}\cup K_p,\eta_p+\frac6p)$--quasi--action
of $\Gamma_i$ on the set $T_p\times Z_p\times\Sym(Z_p)$ and that $(\phit_{i,p})_h=(\rho_p)_h\times\id_{Z_p}\times\id_{\Sym(Z_p)}$
is independent of $i\in\{1,2\}$ and commutes with $\Vt_p$ for every $h\in K_p$.

Let $n(p)=|T_p|\,|Z_p|(|Z_p|!)$.
For $i\in\{1,2\}$ we define the maps
\[
\psi_i:\Gamma_i\to\left(\prod_{p=1}^\infty S_{n(p)}\right)\bigg/\left(\bigoplus_{p=1}^\infty(S_{n(p)},\dist)\right)
\]
by
\begin{align*}
\psi_1(g)&=\big[(\phit_{1,p}(g))_{p=1}^\infty\big] \\
\psi_2(g)&=\big[(\Vt_p\phit_{2,p}(g)\Vt_p^*)_{p=1}^\infty\big].
\end{align*}
Since the $R_{i,p}$ are increasing in $p$ and exhaust $\Gamma_i$, and since the $K_p$ are increasing in $p$ and exhaust $H$,
it follows that $\psi_1$ and $\psi_2$ are group homomorphisms that agree on $H$.
The universal property for amalgamated free products yields a group homomorphism
\[
\psi:\Gamma\to\left(\prod_{p=1}^\infty S_{n(p)}\right)\bigg/\left(\bigoplus_{p=1}^\infty(S_{n(p)},\dist)\right)
\]
that extends $\psi_1$ and $\psi_2$.
To be able to apply Proposition~\ref{prop:soficlift} to conclude that $\Gamma$ is sofic,
it remains to see that for every $g\in\Gamma\backslash\{e\}$, there are $\psi_p(g)\in S_{n(p)}$ such that
$\psi(g)=[(\psi_p(g))_{p=1}^\infty]$ and
\begin{equation}\label{eq:dist1}
\lim_{p\to\infty}\dist(\psi_p(g),\id_{n(p)})=1.
\end{equation}

For $g\in\Gamma$ a nontrivial group element, either $g\in H$ or we may write $g$ as a reduced word
$g=g_1g_2\cdots g_n$ with $g_j\in\Gamma_{i_j}\backslash H$ and $i_1\ne i_2,\,i_2\ne i_3,\ldots,i_{n-1}\ne i_n$.
\begin{enumerate}[(a)]
\item\label{enum:distcase1}
If $g\in H$ or if $n=1$ and $i_1=1$, then we may take $\psi_p(g)=\phit_{1,p}(g_1)$ and we get~\eqref{eq:dist1}
by the corresponding property for the $\phit_{1,p}$.
\item\label{enum:distcase2}
If $n=1$ and $i_1=2$, then we may take $\psi_p(g)=\Vt_p\phit_{2,p}(g_1)\Vt_p^*$ and we get~\eqref{eq:dist1}
by the corresponding property for the $\phit_{2,p}$, because $\dist$ is invariant under left and right multiplication.
\item\label{enum:distcase3}
If $n$ is even and $i_1=1$, then we may take
\begin{align*}
&\psi_p(g)= \\
&\phit_{1,p}(g_1)\Vt_p\phit_{2,p}(g_2)\Vt_p^*\phit_{1,p}(g_3)\Vt_p\phit_{2,p}(g_4)\Vt_p^*
\cdots\phit_{1,p}(g_{2n-1})\Vt_p\phit_{2,p}(g_{2n})\Vt_p^*
\end{align*}
and the asymptotic vanishing of the moment~\eqref{eq:Vmom} as $p\to\infty$ implies that~\eqref{eq:dist1} holds.
\item
In all other cases, the nontrivial element $g$ is conjugate in $\Gamma$ to an element $g'$ of the sort considered in
parts~(\ref{enum:distcase1}), (\ref{enum:distcase2}) or (\ref{enum:distcase3});
say $g=fg'f^{-1}$ for $f\in\Gamma$.
Letting $f_p\in S_{n(p)}$ be any elements so that $\psi(f)=[(f_p)_{p=1}^\infty]$, we may take
$\psi_p(g)=f_p\psi_p(g')f_p^{-1}$.
Since $\dist$ is invariant under left and right multiplication in symmetric groups,
we get~\eqref{eq:dist1} from the same property for the lift $(\psi_p(g'))_{p=1}^\infty$ of $g'$.
\end{enumerate}
\end{proof}

\begin{remark}\label{rem:cheap}
Consider the proof of Theorem~\ref{thm:overamen} in the case of $H$ a finite group.
Here, with a bit of tweaking, we may arrange that $T_p=H$ and that $(\rho_p)_h\in\Sym(H)$ is left multiplication
by $h$, for all $p$.
Now this proof is analogous in spirit to the construction found in~\cite{BDJ08}
of matricial microstates in a tracial free product $A*_DB$ of von Neumann algebras with amalgamation over 
a finite dimensional subalgebra $D$:
one starts with microstates for generators of $A$ and of $B$, one arranges that these microstates agree
on generators of $D$, and then one conjugates with a random unitary that is Haar distributed in the
group of all unitaries in the commutant of $D$.
Where the analogy breaks down, however is that in the proof of Theorem~\ref{thm:overamen},
although we do conjugate with a random permutation that commutes with the action of $H$,
we do not require it to take all values in the commutant of $H$.
Thus, we construct the quasi--actions of $\Gamma_1*_H\Gamma_2$ more cheaply than we would have expected
by analogy with the proof found in~\cite{BDJ08}.
\end{remark}

{}From Theorem~\ref{thm:overamen}, using a well known picture of the HNN extension (which, for convenience, we sketch)
and a result of Elek and Szab\'o about amenable
extensions of sofic groups, we obtain the following
result for HNN extensions of sofic groups.

\begin{cor}
If $\Gamma=G*_\theta$ is an HNN extension of a sofic group $G$ relative to an injective group homomorphism $\theta:H\to G$
where $H$ is a monotileably amenable subgroup of $G$, then $\Gamma$ is sofic.
\end{cor}
\begin{proof}
The group $\Gamma$ is generated by $G$ and an extra generator $t$ with the added relations $t^{-1}ht=\theta(h)$ for all $h\in H$.
As is well known, and as can be proved using Britton's Lemma and the normal form for HNN extensions (see~\cite{LS77}),
the group $\Gamma$ is isomorphic to the crossed product group $K\rtimes_\alpha\Ints$, where $K$ is the subgroup of $\Gamma$ generated by $\bigcup_{k\in\Ints}t^{-k}Gt^k$,
by the automorphism $\alpha:x\mapsto t^{-1}xt$ of $K$.
Moreover, $K$ is a direct limit of groups that are obtained as free products with amalgamation over $H$.
For integers $p$ and $q$, let $K_{[p,q]}$ be the subgroup of $\Gamma$ generated by $\bigcup_{p\le k\le q}t^{-k}Gt^k$.
If $p\le k\le q$, let $\lambda_k:G\to K_{[p,q]}$ denote the injective $*$--homomorphism $g\mapsto t^{-k}gt^k$.
Then we have
\[
K_{[p,q+1]}\cong K_{[p,q]}*_HG,
\]
where the amalgamation is with respect to the maps $\lambda_q\circ\theta:H\to K_{[p,q]}$
and the inclusion map $H\to G$, whereas
\[
K_{[p-1,q]}\cong G*_HK_{[p,q]}\,,
\]
where the amalgamation is with respect to the maps $\lambda_q\restrict_H:H\to K_{[p,q]}$
and $\theta:H\to G$.
By repeated application of Theorem~\ref{thm:overamen}, each $K_{[p,q]}$ is sofic, so their direct limit $K$ is sofic.
Since $K$ is a normal subgroup of $\Gamma$ with infinite cyclic quotient, by Theorem~1 of~\cite{ES06},
$\Gamma$ is sofic.
\end{proof}

\section{Asymptotic freeness}
\label{sec:free}

In~\cite{N93}, A.\ Nica proved asymptotic $*$--freeness for independent random permutation matrices.
Let $I$ be a set and for each $d\in\Nats$, let $(U_i)_{i\in I}$ be an independent family of permutation matrix valued
random variables, where each $U_i=U_{i,d}$ is a uniformly distributed random $d\times d$ permutation matrix.
Let $\Eb$ denote the expectation of the underlying probability space.
Let $F_I=\langle x_i\mid i\in I\rangle$ be the free group with free generators $(x_i)_{i\in I}$
and if $w\in F_I$, let $w(U)$ denote the $d\times d$ permutation matrix obtained by replacing each $x_i$ in $w$ with $U_i$
and each $x_i^{-1}$ with $U_i^*$.
(Of course, if $w$ is the identity element of $F_I$, then $w(U)$ denotes the $d\times d$ identity matrix.)
Nica's asymptotic freeness result is that for every nontrivial $w\in F_I$, we have $\lim_{d\to\infty}\Eb(\tr_d(w(U)))=0$.

The asymptotic vanishing of moments result, Theorem~\ref{thm:BU}, is redolent of asymptotic $*$--freeness.
We will combine it with Nica's asymptotic freeness result to obtain actual
asymptotic $*$--freeness of independent random permutation matrices
and certain families of non--random permutation matrices.
Though, for convenience, our statements are in terms of sequences of $d\times d$ permutation matrices for {\em all} natural numbers $d$,
of course the analogous statements hold for $d_k\times d_k$ matrices, so long as $d_k\to\infty$ as $k\to\infty$.

We consider certain families of sequences of non--random permutation matrices;
for example, these can be taken from from quasi--actions of a group that
are sufficient to demonstrate that the group is sofic.
Let $J$ be a set and suppose for each $j\in J$ and $d\in D$, $B_j=B_{j,d}$ is a $d\times d$ (non--random) permutation matrix.
Suppose 
\[
\forall j\in J\qquad\lim_{d\to\infty}\tr_d(B_{j,d})=0
\]
and
\begin{alignat}{3}
&\forall j_1,j_2\in J
&\quad&\text{either}\quad
&&\lim_{d\to\infty}\dist(B_{j_1}B_{j_2},\id_d)=0 \label{eq:eitheror1} \\
&&\quad&\text{or}\quad
&&\exists j_3\in J\quad\dist(B_{j_1}B_{j_2},B_{j_3})=0, \label{eq:eitheror2}
\end{alignat}
where we are identifying permuation matrices with their corresponding
permutations in $S_d$.

\begin{thm}\label{thm:free}
Let $(U_i)_{i\in I}$ and $(B_j)_{j\in J}$ be as described above.
Then the family
\[
\big(\{U_i,U_i^*\}\big)_{i\in I},\;\{B_j\mid j\in J\}
\]
is asymptotically free as $d\to\infty$, meaning, that we have
\begin{equation}\label{eq:afree}
\lim_{d\to\infty}\Eb\big(\tr_d\big(w_0(U)B_{j_1}w_1(U)B_{j_2}\cdots w_{n-1}(U)B_{j_n}w_n(U)\big)\big)=0
\end{equation}
whenever $n\ge0$, $j_1,\ldots,j_n\in J$, $w_0,w_1,\ldots,w_n\in F_I$, $w_1,\ldots,w_{n-1}$ are nontrivial words and
if $n=0$ then $w_0$ is nontrivial.
\end{thm}
\begin{proof}
Using the properties of the trace and the property~\eqref{eq:eitheror1}--\eqref{eq:eitheror2} of the family of the $B_j$,
we may cyclically reduce any expression of the form appearing on the left--hand--side of~\eqref{eq:afree}
and we see that it equals an expression in one of the three forms
\begin{gather}
\lim_{d\to\infty}\Eb(\tr_d(w_1(U))) \label{eq:afree1} \\
\lim_{d\to\infty}\Eb(\tr_d(B_{j_1})) \label{eq:afree2} \\
\lim_{d\to\infty}\Eb\big(\tr_d\big(B_{j_1}w_1(U)B_{j_2}\cdots w_{n-1}(U)B_{j_n}w_n(U)\big)\big) \label{eq:afree3}
\end{gather}
where $j_1,\ldots,j_n\in J$ and $w_1\ldots,w_n$ are nontrivial elements of $F_I$.
Here we use that if $C_d$ and $D_d$ are permutation matrices and if $\lim_{d\to\infty}\dist(C_d,D_d)=0$,
then for any permutation matrix $V$,
we have $\lim_{d\to\infty}\tr_d(VC_d-VD_d)=0$.

The limit in~\eqref{eq:afree1} vanishes by Nica's asymptotic freeness result.
The limit in~\eqref{eq:afree2} vanishes by hypothesis.
For the limit in~\eqref{eq:afree3}, we will use Nica's asymptotic freeness result and Theorem~\ref{thm:BU}.
Let $V$ be a uniformly distributed random permutation matrix that is independent from all the $U_i$.
Since the distribution of the family $(VU_iV^*)_{i\in I}$ is the same as for $(U_i)_{i\in I}$,
it will suffice to show
\begin{equation}\label{eq:BVU}
\lim_{d\to\infty}\Eb\big(\tr_d\big(B_{j_1}Vw_1(U)V^*B_{j_2}\cdots Vw_{n-1}(U)V^*B_{j_n}Vw_n(U)V^*\big)\big)=0.
\end{equation}
{}From Nica's asymptotic freeness result, we get for every $\eps>0$
\[
\lim_{d\to\infty}\Pb\big(\max_{1\le j\le n}\tr_d(w_j(U))\ge\eps\big)=0,
\]
where $\Pb$ means the probability of the event.
Therefore, we can find a sequence $\eps_d\searrow0$ such that $\lim_{d\to\infty}\Pb(F_d)=0$,
where $F_d$ is the event
\[
\max_{1\le j\le n}\tr_d(w_j(U))\ge\eps_d\,.
\]
Since $V$ and $(U_i)_{i\in I}$ are independent, we can evaluate the expectation in~\eqref{eq:BVU} by first, for
each fixed choice of values for $(U_i)_{i\in I}$, integrating with respect to $V$, and then integrating
with respect to the $(U_i)_{i\in I}$.
For any choice of $(U_i)_{i\in I}$, we have by a trivial bound 
\[
\int\tr_d\big(B_{j_1}Vw_1(U)V^*B_{j_2}\cdots Vw_{n-1}(U)V^*B_{j_n}Vw_n(U)V^*\big)\,dV\le1.
\]
If we choose values of $(U_i)_{i\in I}$ that lie in the complement of the event $F_d$, then by Theorem~\ref{thm:BU},
letting $f(d)=\max(\tr_d(B_{j_1}),\tr_d(B_{j_2}),\ldots,\tr_d(B_{j_n}))$,
we have
\begin{multline*}
\int\tr_d\big(B_{j_1}Vw_1(U)V^*B_{j_2}\cdots Vw_{n-1}(U)V^*B_{j_n}Vw_n(U)V^*\big)\,dV \\
\le C_n\max(f(d),\eps_d)+D_nd^{-1},
\end{multline*}
where $C_n$ and $D_n$ are the constants from Theorem~\ref{thm:BU}.
So we get the upper bound
\begin{multline*}
\Eb\big(\tr_d\big(B_{j_1}Vw_1(U)V^*B_{j_2}\cdots Vw_{n-1}(U)V^*B_{j_n}Vw_n(U)V^*\big)\big) \\
\le C_n\max(f(d),\eps_d)+D_nd^{-1}+\Pb(F_d),
\end{multline*}
which tends to $0$ as $d\to\infty$.
\end{proof}

\smallskip
\noindent
{\em Note added in proof:}
After this paper was accepted for publication, independent papers by Paunescu~\cite{P} and
Elek and Szab\'o~\cite{ES} appeared,
proving that soficity of groups is preserved under taking free products with amalgamation
over arbitrary amenable groups.
Also (in March, 2011), equation~\eqref{eq:Irsum} and surrounding description were corrected.

\end{document}